\documentclass[11pt,a4paper]{article}
\usepackage{amssymb,comment,amsmath,amsthm}
\usepackage{graphicx}
\usepackage{color}
\usepackage{titling,titlesec}
\usepackage{url}
\usepackage{amsthm,lipsum}
\usepackage{geometry}
\titleformat{\subsection}[runin]
{\normalfont\bfseries}{}{0em}{}[.]
\newcounter{foo}

\newfont{\blb}{msbm10 scaled\magstep1}
\newfont{\comp}{cmr12 scaled\magstep1}
\newfont{\compb}{cmr10 scaled\magstep2}
\newfont{\sbb}{cmssbx10 scaled\magstep3}
\newfont{\sbbb}{cmssbx10 scaled\magstep5}
\newfont{\sbs}{cmssbx10 scaled\magstep1}
\newtheorem{theorem}{Theorem}
\newtheorem{lemma}[subsection]{Lemma}
\newtheorem{claim}[subsection]{Claim}

\newtheorem{definition}[foo]{Definition}

\newcommand{\FF}{\mathbb F}

\newcommand{\vep}{\varepsilon}
\newcommand{\End}{\mathrm{End}}
\newcommand{\Ter}{\mathrm{Ter}}
\newcommand{\good}{\mathrm{Good}}
\newcommand{\Cay}{\mathrm{Cay}}
\allowdisplaybreaks
\title{Rainbow paths in directed digraphs}
\author{
Chase Wilson\thanks{Department of Mathematics, University of California, San Diego, CA, 92093-0112 USA. E-mail: c7wilson@ucsd.edu}}
\begin{document}
\date{ }
\setlength{\droptitle}{-2.5cm}
\maketitle
\vspace{-0.5in}
\begin{abstract}
An old problem in combinatorial group theory asks, given a group $\Gamma$ and a subset $S \subseteq \Gamma$, when does there exist an ordering $s_1, \cdots, s_k$ of the elements of $S$ such that the partial products $\prod_{i = 1}^j s_i$, $1 \leq j \leq k$, are all distinct. If such an ordering exists, we call $S$ rearrangeable. There have been many conjectures about rearrangeable subsets, the most general being that for every group, every subset not containing the identity element is rearrangeable. We prove an asymptotic version of this: For any Group $\Gamma$ and any subset $S \subseteq \Gamma$, there exists a rearrangeable set $S' \subseteq S$ such that $|S'| = |S| - o(|S|)$.
To do this we build upon the work of Buci\'c, Frederickson, M\"uyesser, Pokrovskiy, and Yepremyan focusing on the following problem of independent interest in Graph Theory. If $G$ is a d-regular properly colored directed graph does there exist a rainbow path of length $d - 1$? We establish an asymptotic version of this, proving that $G$ contains a rainbow path of length $d - o(d)$. This solves two problems given by Buci\'c, Frederickson, et al. and proves the above result on rearrangeable subsets of groups by considering the Cayley graph of $\Gamma$ with (not necessarily generating) set $S$.
\end{abstract}
\begin{section}{Introduction}
An old question in combinatorial group theory concerns the ability to order the elements of a subset of a group such that the partial products are all distinct. Given a group $\Gamma$ and a subset $S \subseteq \Gamma$, we want to know if there exists an ordering $s_1, \cdots, s_k$ of $S$ such that the partial products $\prod_{i = 1}^j s_j$ for $(1 \leq j \leq k)$ are all distinct. If such an ordering exists, $S$ is called \textit{rearrangeable}. This concept is closely related to the study of sequenceable groups, introducted by Gordon \cite{Go} in 1961 which are groups $\Gamma$ for which the set $S = \Gamma$ is rearrangeable. Gordon characterized all sequenceable abelian groups. Keedweel \cite{Ke} conjectured that the only non-sequenceable groups are the dihedral groups of order $6$ and $8$ and the quaternion group. This has recently been resolved for all sufficently large groups by M\"uyesser and Pokrovskiy \cite{MP}. A related problem posed by Ringel asks for an ordering $a_1, \cdots, a_n$ of the elements of a group $\Gamma$ such that the partial products $a_1, a_1 a_2, \cdots, a_1 \cdots a_{n- 1}$ are all distinct and the total product $a_1 \cdots a_n$ is the identity. This has also been resolved by M\"uyesser and Pokrovskiy \cite{MP} for sufficently large groups.

\medskip

The study of rearrangeable subsets began with a conjecture by Graham \cite{Gr} in 1971, which states that for the cyclic group $\mathbb F_p$, every subset $S \subseteq \mathbb F_p \setminus \{ 0 \}$ is rearrangeable. Despite considerable effort, this conjecture remains open for $S$ that are not either very large or very small. The Combinatorial Nullstellensatz can be used to prove the conjecture when $|S| \leq 12$. Kravitz \cite{Kr} and Sawin \cite{Sa} recently each independently showed the conjecture is true when $|S| \leq \log p / \log \log p$. Kravitz \cite{Kr} then improved this to $|S| \leq e^{ (\log p)^{1/4}}$. On the other end of the spectrum, direct constructions show the conjecture holds when $|S| \geq p - 3$ and a recent breakthrough result by M\"uyesser and Pokrovskiy \cite{MP} implies the conjecture holds when $|S| = p - o(p)$.

\medskip

Generalizing Graham's conjecture, Archdeacon, Dinitz, Mattern, and Stinson \cite{ADMS} conjectured that for any group $\Gamma$, every subset $S \subseteq \Gamma \setminus \{e \}$ is rearrangeable. Alspach \cite{Al} proposed an even stronger conjecture, stating that if the product of the elements of $S$ is not the identity, then an ordering exists such that all partial products are distinct and not the identity. Many other slight variations on this, for example by Costa, Morini, Pasotti \cite{CMPP}, and by Costa, Fiore, Ollis and Rovner-Frydman \cite{CDFO} have been conjectured. Buci\'c, Frederickson, M\"uyesser, Pokrovskiy, and Yepremyan \cite{BFMPY} showed that for any group $\Gamma$ and any subset $S \subseteq \Gamma$, there exists an ordering of the elements of $S$ such that all but $o(|S|)$ of the partial products are distinct. They also proved the much stronger asymptotic result that for any group $\Gamma$ and any subset $S \subseteq \Gamma$, there exists a rearrangeable subset of $S$ of size $|S| - o(|S|)$ in each of the following cases
\begin{itemize}
\item $\Gamma = \mathbb F_2^k$ for some $k$.
\item $\Gamma = \mathbb F_p$ for some prime $p$ and $|S| \geq p^{3/4 + o(1)}$.
\item $|S| = \Omega(|\Gamma|)$
\end{itemize}
 In a breakthrough result \cite{BBK}, Bedert, Buci\c, Kravitz, Montgomery, and M\"uyesser proved the full Graham's conecture for the group $\mathbb{F}^k_2$ \cite{BBK}. They also showed that Graham's conjecture holds for arbitrary groups when $S$ has size at least $|\Gamma|^{1- c}$ for some absolute constant $c$ depending only on $\Gamma$. To do this they used a somewhat similar approach of Buci\'c, Frederickson, M\"uyesser, Pokrovskiy, and Yepremyan \cite{BFMPY} combined with the absorbtion method. Shortly after that, Pham and Sauermann, using a completly different probabilistic approach showed that for every $\alpha \in (0, 1)$ the conjecture holds in $\Gamma = \mathbb{F}_p$ and $|S| \leq p^{1 - \alpha}$ as long as $S$ is sufficiently large as a function of $\alpha$ \cite{PS}. These results together completly resovel Graham's original conjecture in $\mathbb{F}_p$ for $p$ sufficiently large. 

\medskip

While the case of Graham's conjecture for $\Gamma = \mathbb{F}_p$ is now resolved, the more general case where $\Gamma$ is an arbitrary finite group is still wide open. In this paper, we make progress towards the general case by building upon the techniques of Buci\'c, 
Frederickson, M\"uyesser, Pokrovskiy, and Yepremyan and showing that their strong asymptotic result holds with no assumptions on $S$ or $\Gamma$.

\medskip

\begin{theorem} \label{main_group} For any finite group $\Gamma$ and any subset $S \subseteq \Gamma$, there exists a rearrangeable subset $S' \subseteq S$ such that $|S'| = |S| - o(|S|)$.
\end{theorem}

\medskip

To prove this we translate this problem to finding large paths in regular properly edge-colored directed graphs. Specifically, consider the Cayley Graph $\Cay(\Gamma, S)$ whose vertices are elements of $\Gamma$ and where there is an $s$-colored directed edge $(g, gs)$ for each $g \in 
\Gamma$ and $s \in S$. Then for any $s_1, \cdots, s_k \in S$, the partial products $\prod_{i = 1}^j s_j$ for ($1 \leq j \leq k$) are all distinct iff there exists a path in $\Cay(\Gamma, S)$ of the form $( g s_1, g s_1 s_2, \cdots, g s_1 \cdots s_k)$.

\medskip

The problem of finding long rainbow paths in graphs is not new. Hahn \cite{Ha} conjectured in 1980 that there exists a rainbow path of length $n - 1$ in any proper-edge coloring of the complete graph $K_n$. This was disproved by Maamoun and Meyniel \cite{MM}, who showed that the Cayley graph $\Cay (\FF_2^k, \FF_2^k \setminus \{0\})$ serves as a counterexample. Andersen \cite{An} subsequently conjectured that a rainbow path of length $n - 2$ always exists. While this remains open, Alon, Pokrovskiy, and Sudakov \cite{APS} proved the existence of rainbow paths of length $n - o(n)$.

\medskip

For $d$-regular graphs, Schrijver \cite{Sc} conjectured the existence of a rainbow path of length $d - 1$ in any properly edge-colored $d$-regular graph, proving it for $d \leq 10$. Johnston, Palmer and Sarkar \cite{JPS} showed a path of length at least $2d/3$ exists. Babu, Chandraw, and Rajendraprasad \cite{BCR} proved the existence of a rainbow path of length $d - o(d)$ if the graph is $C_4$-free, and length $d - 2$ if the girth is $\Omega( \log d)$. Buci\'c, Frederickson, M\"uyesser, Pokrovsiky, and Yepremyan \cite{BFMPY}, motivated by the rearrangeable set problem, established the existence of a rainbow path of length $d - o(d)$ in any properly $d$-edge-colored $d$-regular undirected graph. For directed graphs, they showed the same $d - o(d)$ bound holds, but only under the restriction that $d = \Omega( |V(G)|)$. They posed the problem of removing this restriction.

\medskip

Our main graph-theoretic result resolves this question and as a consequence proves Theorem \ref{main_group}.
\begin{theorem} \label{main} Let $G$ be a $d$-regular properly edge colored directed graph $G$. Then $G$ contains a rainbow path of length $ d - o(d)$.
\end{theorem}
\end{section}
\begin{section}{Notation and Terminology}
Directed graphs, for our purposes, are loopless. For a directed graph $G$ and vertex $v \in V(G)$, we let $N^+(v)$ and $N^-(v)$ denote the out neighborhood of $v$ and the in neighborhood of $v$ respectively. Similarly for a subset $S \subseteq V(G)$, we let $N^+(S) = \bigcup_{v \in S} N^+(S)$ and $N^-(S) = \bigcup_{v \in S} N^-(S)$.

\medskip

For a graph $H$, we let $V(H)$ denote the vertex set of $H$, $E(H)$ denote the edge set of $H$, and for a collection $\mathcal H$ of graphs, we let $V(\mathcal H) = \bigcup_{H \in \mathcal H} V(H)$.

\medskip

If $H$ is an edge-colored graph and $e \in E(H)$ we let $C(e)$ denote the color of $e$. We let $C(H) = \bigcup_{e \in E(H)} C(e)$ and if $\mathcal H$ is a collection of edge-colored graphs, we let $C(\mathcal H) = \bigcup_{H \in \mathcal H} C(H)$. For an edge-colored graph $H$, a vertex set $V_0 \subseteq V(H)$, and a color set $C_0 \subseteq C(H)$, we let $H[V_0, C_0]$ denote the subgraph of $H$ with vertex set $V_0$ and edge set
\[
\{ (u, v) \in E(H) : u, v \in V_0, C(u, v) \in C_0 \}
\]
For a graph $G$ and vertex $v \in V(G)$, we let $\deg^+_G(v)$ and $\deg^-_G(v)$ denote the out and in degrees of $v$ respectively. We let $\delta^+(G)$ be the minimum out-degree among vertices in $G$ and $\Delta^-(G)$ be the maximum in-degree among vertices in $G$.

\medskip

We say $\alpha \ll \beta$ to mean $\alpha$ is sufficiently small compared to $\beta$, the exact dependance can be extracted from whatever argument follows.

\medskip

For a path $P$ we let $\End(P)$ denote the endpoint of $P$. We say that a path $Q$ is a prefix of $P = (v_1, \cdots, v_m)$ if $Q = (v_1, \cdots, v_k)$ for some $k \leq m$ and that $P$ extends $Q$ if $Q$ is a prefix of $P$.
\end{section}
\begin{section}{Proof Overview} \label{proof_overview}
In this Section we state the main Lemmas and use them to prove Theorem \ref{main}.

\medskip

Our main contribution is essentially reducing Theorem \ref{main} to the case where $d = \Omega(n)$. In particular we prove
\begin{theorem} \label{sec_1_main} If $G$ is an $n$-vertex properly-colored $d$-regular directed graph with no rainbow path of length $(1 - \vep)d$, then for every $\gamma > 0$ there exists a $K > 0$, a rainbow path $P$ and an induced subgraph $G' \subseteq G [ V(G) \setminus V(P), C(G) \setminus C(P) ]$ such that if $d' = d - |E(P)|$,
\begin{itemize}
	\item For every $u \in V(G')$, there exists a rainbow path extending $P$ by length at most $K$ that ends at $u$.
	\item $|V(G')| \leq K d'$
	\item $\delta^+(G') \geq (1 - \gamma) d'$
\end{itemize}
\end{theorem}
From here we use ideas introduced in \cite{BFMPY} to find a rainbow path of length close to $d'$ in $G'$ and connect it to $P$ to form a path of length at least $(1 - \vep) d$. To do this we first use the following Lemma to find a large induced subgraph $G''$ of $G'$ with high min degree and good expansion properties.
\begin{definition} Let $G$ be a directed graph on $n$ vertices. For $U \subseteq V(G)$ and $\nu > 0$, we define the $\nu$ robust out neighborhood of $U$ to be
\[
RN^{+}_{\nu, G}(U) = \{ v \in V(G) : |n^{-}(v) \cap U| \geq \nu n \}
\]
$G$ is a robust $(\nu, \tau)$-out-expander if every $U \subseteq V(G)$ with $\tau n \leq |U| \leq (1 - \tau) n$ satisfies
\[
|RN^{+}_{\nu, G} (U) \setminus U | \geq \nu n
\]
\end{definition}
\begin{lemma} \label{expander_finder} For any $\delta, \tau, \alpha$, there exists $\nu$ such that for every positive integer $n$ and every $n$-vertex digraph $G$ with $\delta^{+}(G) \geq \alpha n$, there exists a non-empty induced subgraph $G'$ of $G$ satisfying:
\begin{enumerate}
\item[\textbf{A1}] $G'$ is a robust $(\nu, \tau)$-out-expander.
\item[\textbf{A2}] Every vertex $v \in V(G')$ satisfies $\deg^{+}_{G'}(v) \geq \deg^+_{G}(v) - \delta n$
\end{enumerate}
\end{lemma}
Working in $G''$ and using the next Lemma, we are able to set aside a small set of vertices and colors capable of connecting a small number of (vertex and color) disjoint rainbow paths.
\begin{lemma} \label{path_connecter} Let $\nu, \tau, \alpha, p, 1/M \leq 1$ be positive constants satisfying $\nu + \tau \leq \alpha$. Then for all sufficiently large $n$, the following holds: Let $G$ be a properly edge-colored directed graph on $n$ vertices. Suppose that $G$ is a robust $(\nu, \tau)$-out-expander with $\delta^+(G) \geq \alpha n$. Let $V_0, C_0$ be independent $p$-random subsets of $V(G), C(G)$, respectively. Then with high probability $1 - o_n(1)$ the following all hold
\begin{enumerate}
\item[ \textbf{B1}] $|V_0| \leq 2pn$
\item[ \textbf{B2}] Every vertex $v \in V(G)$ satisfies
\[
\deg^+ (v, V \setminus V_0 ; C \setminus C_0 ) \geq (1 - 3p) \deg^+(v)
\]
\[
\deg^{+} (v, V_0, C_0) \geq  p^2 \deg^{+}(v)/3
\]

\item[ \textbf{B3}] For any two distinct vertices $u, v \in V(G)$ with $\deg^{-}(v) \geq  (\tau + \nu) n$ and for any vertex subset $V_1 \subseteq V_0$ and color subset $C_1 \subseteq C_0$ each of size at most $M$, there exists a rainbow directed path of length at most $\nu^{-1} + 1$ from $u$ to $v$ in $G$ whose internal vertices are in $V_0 \setminus V_1$ and whose colours are in $C_0 \setminus C_1$.
\end{enumerate}
\end{lemma}
We then use the following Lemma to find a constant number of disjoint rainbow paths starting at high in-degree vertices with total length at least $(1 - \vep) d'$ and connect them using Lemma \ref{path_connecter}.
\begin{lemma} \label{path_forest} Let $\vep > 0$ and $G$ be an $n$-vertex properly colored directed graph  with $\delta^+(G) \geq d$ where $1/d \ll \vep$. Then for every $\vep > 0$ and vertices $V = \{v_0, \cdots, v_{4/\vep^2} \}$, there exists a family of vertex-disjoint, color-disjoint paths $\mathcal P = \{ P_0, \cdots, P_{4/\vep^2} \}$ such that the total length of all paths is at least $(1 - \vep) d$ and the first vertex in $P_i$ is $v_i$ for all $i$.
\end{lemma}
Putting this all together gives us a rainbow path of length at least $(1 - \vep) d$, proving Theorem \ref{main}. The full details are in Section \ref{main_proof}.
\end{section}
\begin{section}{Finding a dense subgraph}
In this Section we prove Theorem \ref{sec_1_main}. For the remainder of this Section we fix $\vep > 0$ and, with $d$ sufficiently large, a $d$-regular properly edge-colored directed graph $G$ and suppose that $G$ contains no path of length $(1 - \vep)d$. We will also use ``path'' to mean ``directed rainbow path'' and assume all paths start at the same fixed vertex $v$ unless stated otherwise.

\medskip

We also fix a constant $\vep_1 \ll \vep$.
\begin{definition} We say a path $P$ is terminal if $|N^+(\End(P)) \cap V(P)| \geq \vep_1 d$ and this condition does not hold for any prefix of $P$.
\end{definition}
Since we assume that Theorem \ref{main} is false, if we start with the empty path and keep extending it, we eventually reach a terminal path since otherwise as long as $|P| \leq (1 - \vep_1) d$, we could find a $u \in N^{+}(\End(P))$ with the color of $(\End(P), u)$ not in $C(P)$ and $u \not \in V(P)$ and so we could extend $P$.
\begin{definition} For a path $P$, a path $P'$ is a \textit {child} of $P$ if it extends $P$ by length $1$. $P'$ is the \textit{parent} of $P$ if $P$ is a child of $P'$.
\end{definition}
\begin{definition} We say a terminal path $P$ is a strong-terminal path if at least $\vep_1 d$ children of the parent of $P$ are also terminal and a weak-terminal path otherwise. .
\end{definition}
We will ignore weak terminal paths By this we mean that we will again narrow our definition of path to only include paths that are not weak-terminal paths and do not extend any terminal path by length at least 1. Therefore any non-terminal path $P$ has at least $d - |E(P)| - 2 \vep_1 d$ children: To see this observe there are $d - |E(P)|$ out-edges of $\End(P)$ that are not a color in $C(P)$, there are at most $\vep_1 d$ out-neighbors of $\End(P)$ that are in $V(P)$ and there are at most $\vep_1 d$ out-neighbors $u$ of $\End(P)$ for which $P \cup \{u \}$ is a weak terminal path.

\medskip

We let $\Ter$ denote the set of strong-terminal paths.

\medskip

Let us describe a very broad overview of the strategy in this Section supposing that $|\Ter| = O(d)$. For every parent $P$ of a strong-terminal path, each of its $\Omega(d)$ children that are terminal has out-degree $O(d)$ into $P$. Therefore $\Ter$ has $\Omega(d^2)$ out-degree into every strong-terminal path. Since the max in-degree is $d$, each strong-terminal path has $\Omega(d)$ vertices with $\Omega(d)$ in neighbors of $\Ter$ which we will call \textit{sink vertices}. But since $\Ter$ only sends $O(d^2)$ edges out, there are at most $O(d)$ \textit{sink vertices} in total. So there exists a set vertices $U$ of size $O(d)$  such that each strong-terminal path intersects $U$ $\Omega(d)$ times.

\medskip

From this one can deduce that there exists a path $P$ such that every constant length extension of $P$ intersects $U$ many times. None of these extensions can get too far away from $U$ until they intersect $U$ the requisite number of times and we are able to extract a subgraph from these extensions that satisifes Theorem \ref{sec_1_main}.

\medskip

In general $|\Ter| = O(d)$ is a very strong condition that we cannot expect in general to hold, however as described in the following Definition and Lemma, we can replace parents of strong-terminal paths with a specific cover $C$, and the vertices at the ends of strong-terminal paths with sets $ \{ U_{P' } \}_{P' \in C}$.
\begin{definition} For a path $P$, a collection $C$ of paths extending $P$ is a cover of $P$ if every strong terminal path extending $P$ is an extension of some path in $C$.
\end{definition}
For a path $P$ and a path $P'$ extending $P$, we let $(P' - P)$ denote the suffix of $P'$ starting at $\End(P)$.
\begin{lemma} \label{short_ends} There exists constants $1/K, \vep_2, \vep_3 \ll \vep_1$, a path $P$, a cover $C$ of $P$ and for every $P' \in C$, a set $U_{P'} \subseteq V(G)$ such that
\begin{enumerate}
\item[\textbf{D1}] $| \bigcup_{P' \in C} U_{P'} | \leq Kd$
\item[\textbf{D2}] $ |U_{P'} | \geq \vep_3 d$ for every $P' \in C$.
\item[\textbf{D3}] For every $P' \in C$ and $u \in U_{P'}$
\[ | N^{+} (u) \cap V(P' - P) | \geq \vep_2 d
\]
\end{enumerate}
\end{lemma}
\begin{proof} We attempt to define a sequence of covers $S_0, S_1, \cdots, S_{1/2\vep + 1}$ of the singleton path $\{v\}$, and families of sets of vertices, $ \{V_{0, P}\}_{P \in S_0}, \cdots, \{V_{1/2\vep + 1, P} \}_{P \in S_{1/2 \vep + 1}}$ where $ V_{i, P} \subseteq \Ter(P)$. We define these sets inductively as follows: We let $S'_0$ be the set of parents of strong terminal paths, we let $S_0$ be a minimal subcover of $S'_0$ and we let $V_{0, P}$ be the endpoints of strong terminal paths that are children of $P$. We maintain that
\begin{enumerate}
\item[$\textbf{C1}_i$] Every path in $S_i$ has length at most $d - i \cdot \vep_1 d /2 $
\item[$\textbf{C2}_i$] For every $P \in S_i$, $| V_{i, P} | \geq \vep_1 d/2^i$
\item[$\textbf{C3}_i$] For every $P \in S_i$, and $u \in V_{i, P}$, $|N^+(u) \cap P| \geq \vep_1 d - i \cdot \vep_2 d$.
\end{enumerate}
By the definition of strong terminal path, $(S_0, \{V_{0, P} \}_{P \in S_0})$ satisfies $\textbf{C1}_0$, $\textbf{C2}_0$, $\textbf{C3}_0$.
We show that if at any point, we cannot find a pair $(S_i, \{V_{i, P}\}_{P \in S_i})$ that satisfies the above requirements then we can find a triple $(P, C, \{U_{P'} \}_{P' \in C})$ that satisfies \textbf{D1}, \textbf{D2}, \textbf{D3}. Note this must happen for some $i \leq 2/\vep_1$ since otherwise $\textbf{C1}_i$ forces the length of paths in $S_i$ to be negative.

\medskip

For a path $P$ and $u \in \Ter(P)$, say that $u$ is $i$-bad for $P$ if $| N^+(u) \cap V(P)| \geq \vep_1 d - \vep_2 i \cdot d$ and $i$-good if it is not $i$-bad.

\medskip

We let
\[
A(P) = \bigcup_{P' \in S_{i - 1}: P' \text{ extends } P } V_{i - 1, P'}
\]
and we say a path is heavy if $|A(P)| \geq K d$.

\medskip

If the singleton path $\{v\}$ is not-heavy then the triple $(\{v\}, S_{i - 1}, \{V_{i - 1, P} \}_{P \in S_{i - 1}} )$ satisfies \textbf{D1}, \textbf{D2}, \textbf{D3} for $\vep_3 = \vep_1 /2^i \geq \vep_1 / 2^{2/\vep_1}$ as long as $\vep_2 < \vep_1 - 2 \vep_2/\vep_1 $. In this case we would be done so suppose that $\{v \}$ is heavy.

\medskip

We can now define $S_i$ and $\{V_{i, P} \}_{P \in S_i}$. We set $S_i$ to be the set of paths which are not heavy but their parent is heavy and we set $V_{i, P}$ to be the set of elements in $ A(P) $ that are $i$-bad for $P$.

\medskip

It remains to check $\textbf{C1}_i$, $\textbf{C2}_i$, $\textbf{C3}_i$. $\textbf{C3}_i$ holds by definition of $i$-bad paths.

\bigskip

We turn our attention to verifying $\textbf{C1}_i$. Suppose towards contradiction that $\textbf{C1}_i$ does not hold. Then there is a heavy path of length at least $d - i \cdot \vep_1 d /2 - 1$. Since each path in $S_{i - 1}$ has length at most $\vep d - (i - 1) \cdot \vep_1 d /2 d$, by $\textbf{C3}_{i - 1}$ each element in $A(P)$ has at least
\[
\vep_1 d - i \cdot \vep_2 d - \vep_1 d/2 - 1 > \vep_1 d/ 4
\]
out-neighbors neighbors in $V(P)$. But since $P$ has length at most $d$ and each vertex has indegree at most $d$ we have have
\[
(\vep_1 d/4) (K d) < d^2
\]
which is a contradiction since $1/K \ll \vep_1$. This verifies $\textbf{C1}_i$.

\bigskip

All that remains is to verify $\textbf{C2}_i$.
\begin{claim} If $\textbf{C2}_i$ fails, then there exists a triple $(P, C, \{U_{P}\}_{P \in C})$ satisfying \textbf{D1}, \textbf{D2}, \textbf{D3}.
\end{claim}
\begin{proof} Suppose that $\textbf{C2}_i$ fails. Then there are at most $\vep_1 d / 2^i$ vertices in $A(P)$ that are $i$-bad for $P$. Let $C$ be the set of paths in $S_{i - 1}$ that extend $P$ and for $P'' \in C$ define $U_{P'}$ to be the set of vertices in $V_{i - 1, P'}$ are $i$-good for $P$.
We claim that the triple
\[
(P, C, \{U_{P'} \}_{P' \in C} )
\]
satisfies the \textbf{D1}, \textbf{D2}, \textbf{D3}. Since $S_{i - 1}$ is a cover of the empty path, $C$ is a cover of $P$.

\medskip

\textbf{D1} holds since $A(P) < Kd$.

\medskip

Let $P' \in C$. Then by $\textbf{C2}_{i-1}$,
\[
|V_{i - 1, P'}| \geq \vep_1 d/2^{i - 1}
\]
so if we let $b$ be the number of paths in $V_{i - 1, P'}$ that are i-bad for $P$,
\[
|U_{i, P'}| \geq \vep_1 d/2^{i - 1} - b \geq \vep_1 d/2^i \geq \vep_1 d / 2^{2/\vep_1}
\]
So \textbf{D2} holds with $\vep_3 = \vep_1 2^{-2/\vep_1}$.

\medskip

By $\textbf{C3}_{i -1}$, for every $P' \in C$, $u \in U_{ P'}$,
\[
|N^+(u) \cap V(P'') | \geq \vep_1 d - (i - 1) \vep_2 d
\]
And since $u$ is $i$-good for $P$,
\[
|N^+(u) \cap V(P'' - P') | \geq \vep_2 d
\]
so \textbf{D3} is verified. This proves the claim
\end{proof}
Thus we may assume that $\textbf{C2}_i$ holds.

\medskip

All that remains to be verified is that $S_i$ is indeed a cover of the singleton path $\{v\}$. Observe that by the same argument we used to verify $\textbf{C2}_i$, no path in $S_{i - 1}$ is heavy. Thus, if $T$ is a strong-terminal path, since $S_{i - 1}$ is a cover of $\{v\}$ there is a path $P' \in S_{i - 1}$ which $T$ extends. Since $\{v\}$ is heavy and $P'$ is not heavy, there is a minimal prefix $P$ of $P'$ that is not heavy. By definition $P \in S_i$ and $T$ extends $P$ so $S_i$ is a cover of $\{v\}$.
\end{proof}
\begin{definition} For a path $P$ and a cover $C$ of $P$ say a path $P'$ is in $(P, C)$ if $P' = P'' - P$ for some $P'' \in C$.
\end{definition}
\begin{lemma} \label{heavy_intersecter} Let $(P, C, \{U'\}_{P' \in C})$ be a triple satisfying Lemma \ref{short_ends}. There exists a constants $\vep_4, K_2 > 0$ and a vertex set $B$ with $|B| \leq K_2 d$ such that every path in $(P, C)$ intersects $B$ at least $\vep_4 d$ times.
\end{lemma}
\begin{proof} Let
\[
U = \bigcup_{P' \in C} U'_{P'}
\]
we let, $N_{(P, C)}(v) = N^{-}(v) \cap U$
\[
B = \{v \in V(G) : | N_{(P, C)}(v) | \geq \vep_2 \vep_3 d/2 \}
\]
Since every vertex has outdegree at most $d$, $|B| \leq 2 |U|/(\vep_2 \vep_2) \leq 2K d/(\vep_2 \vep_3) $.
\claim For every path $P'$ in $(P, C)$,
\[
\sum_{v \in P'} | N_{(P, C)} (v) | \geq \vep_3 \vep_2 d^2
\]
\proof For every $P'$ in $(P, C)$ and every $u \in U'_{P'}$,
\[
|N^{+}(u) \cap V(P')| \geq \vep_2 d
\]
Additionally, $|U'_{P'}| \geq \vep_3 d$. Putting these two facts together proves the claim.
Then since every path has length at most $d$ and every vertex has in-degree at most $d$, for every $P'$ in $(P, C)$,
\[
d \vep_2 \vep_3 d /2 + d |V(P') \cap B| \geq \vep_2 \vep_3 d^2
\]
so $|V(P') \cap B| \geq \vep_2 \vep_3 d/2 $ so we can take $\vep_4 = \vep_2 \vep_3/2$.
\end{proof}
\begin{lemma} For every constant $T > 0$, there exists a path $P'$ extending $P$ such that for every path $P''$ extending $P'$ by length $T$
\begin{itemize}
\item $P'' - P'$ intersects $B$ at least $\vep_4 T/2 $ times.
\item No prefix of $P''$ is in $C$.
\end{itemize}
\end{lemma}
\begin{proof} Suppose towards contradiction this is not true. Then we can start at $P' = P$ and iteratively extend $P'$ by
\begin{itemize}
\item If there is a path $P'' \in C$ extending $P'$ by length at most $T$, set $P' = P''$ and terminate.
\item Otherwise we extend $P'$ by length $T$ while intersecting $B$ at most $\vep 4 T/2$ additional times.
\end{itemize}
Eventually we reach a path in $C$ and have intersected $B$ at most least $(\vep_4 T/2) (d/T) + T < \vep_4 d$ times which is false by Lemma \ref{heavy_intersecter}.
\end{proof}
We now have all the ingredients to prove Theorem \ref{sec_1_main}. We prove it in a slightly different form which we state as the following Lemma.
\begin{lemma} Let $P'$ be as in the previous Lemma, let $d' = d - |E(P')|$. For every constant $\gamma > 0$, as long as $T$ is a sufficiently large and $\vep_1 \ll \gamma$, there exists a constant $K_3 > 0$ and a subgraph $G'' \subseteq G[V(G) \setminus V(P'), C(G) \setminus C(P')]$ such that
\begin{itemize}
\item For every $u \in G''$, there exists a path of length at most $T$ in $G'$ from $\End(P')$ to $u$
\item $|V(G'')| \leq K_3 d'$
\item $\delta^{+}(G'') \geq (1 - \gamma - 3\vep_1/\vep) d' $
\end{itemize}
(Note that since we assume Theorem \ref{main} is false, $d' \geq \vep d$ and since $\vep_1 \ll \vep$, the loss of $3 \vep_1/\vep$ is not significant).
\end{lemma}
\begin{proof}
Call a path $Q$ in $G$ good if $Q$ intersects $B$ at most $\vep_4 T/2 - 1$ times and $P' \cup Q$ is a path in $G$.
\claim If $(w_1, \cdots, w_k)$ is a good path and $(w_1, \cdots, w_k)$ intersects $B$ at most $\vep_4T/2 - 2$ times, then there are at least $d' - 2\vep_1 d - T$ out-neighbors $u$ of $w_k$ such that $(w_1, \cdots, w_k, u)$ is a good path.
\proof Let us count the number of out-neigbors $u$ of $w_k$ in $G$ for which $P' \cup \{w_1, \cdots, w_k, u\}$ is not a path in $G$. There are at most $|E(P')| + k < |E(P')| + T$ such $u$ for which $C(w_k, u) \in C ( P' \cup \{w_1, \cdots, w_k\})$. There are at most $\vep_1 d$ such $u$ for which $P' \cup \{w_1, \cdots, w_k, u\}$ is a weak-terminal path. And there are at most $\vep_1 d$ such $u$ for which $u \in V( P' \cup \{w_1, \cdots, w_k \} )$. All other out-neighbors of $w_k$ extend to a path. Thus there are at most $|E(P')| + 2 \vep_1 d + T = (d - d') + 2 \vep_1 d + T$ out-neighbors $u$ in $G$ for which $P' \cup (w_1, \cdots, w_k, u)$ is not a path in $G$. Since $G$ has min degree at least $d$, this proves the claim.

\bigskip

Let $\mathrm{Good}$ be the set of good paths. We now describe how to construct $G''$. Let
\[
H_i = \{ Q \in \good : |V(Q) \cap B| \leq i, \End(Q) \in B \}
\]
and let
\[
B_i = V(H_i) \cap B
\]
Since $B$ has size at most $K_2 d$ and $B_1 \subseteq B_2 \subseteq \cdots \subseteq B_{\vep_4 T/2 - 1} \subseteq B$, there is some index $2 \leq i_0 < \vep_4 T/2 - 1$ such that
\[
|B_{i_0} \setminus B_{i_0 - 1} | \leq \gamma ' d
\]
where $\gamma' = 4 K_2/(T \vep_4)$. Let $B'_{i_0} = B_{i_0} \setminus B_{i_0 - 1}$.

\medskip

We claim we can take $G''$ to be the induced subgraph of $G[V(G) \setminus V(P') , C(G) \setminus C(P')]$ on vertex set $V(H_{i_0}) \setminus B'_{i_0}$.

\medskip

First we show $\delta^+(G'') \geq (1 - \gamma - 3 \vep_1/\vep) d'$. Let $u \in V(H_{i_0}) \setminus B'_{i_0}$. Then there exists a good path $Q$ with $\End(Q) = u$ and $|V(Q) \cap B| < i_0$. By the above claim, $Q$ has a least $d' - 2 \vep_1 - T$ good children. Since $|V(Q) \cap B| < i_0$ each of these good children extends to a path in $H_i$. So if $v$ is the endpoint of a good child of $Q$ and $v \not \in B'_{i_0}$ then $v \in V(H_{i_0}) \setminus B'_{i_0}$ and $(u, v) \in E(G'')$. So
\[
\deg^{+}_{G''}(u) \geq d' - T - 2 \vep_1 d - |B_{i_0} \setminus B_{i_0 - 1}| \geq d' - \gamma' d - 3 \vep_1 d
\]
Thus we can take $\gamma = 2 \gamma'/\vep$.

\medskip

It remains to be shown that there exists a $K_3$ such that $|V(G'')| \leq K_3 d'$. Let $G' = G[V(G), C \setminus C(P)]$. Set $S_0 = B$ and for $i \geq 1$ let $S_i$ be the set of all vertices $w \not \in S_{i - 1}$ such that, in $G'$, all but at most $2 \vep_1 d+ T$ out-neighbors of $w$ are in $\bigcup_{j = 0}^{i - 1} S_{j}$. Let $S'_i = \bigcup_{j = 0}^i S_i$.
\begin{claim} $ |S_i| \leq (2/\vep) |S'_{i-1}|$
\end{claim}
\proof We have
\[
d |S'_{i - 1}| \geq \partial^-_{G'} (S'_{i - 1}) \geq (d' - 2 \vep_1 d - T) |S_i|
\]
which immediately implies the claim as $T$ is constant, $d' \geq \vep d$ and $\vep_1 \ll \vep$.

\medskip

\claim Every good path in $G'$ starting at $\End(P')$ is contained in $S'_T$
\proof Suppose there is such a path $P''$ that reaches a vertex $w \not \in \bigcup_{i = 0}^T S_i$. Then for $1 \leq i \leq |T| - |E(P'')|$ we try to get a path $P''_i = P'' \cup (w_1, \cdots, w_i)$ such that $P' \cup P''_i$ is a path in $G$ and $w_j \not \in S'_{T - j}$ for all $1 \leq j \leq i$.

\medskip

Indeed, assuming this holds for $i -1$, then $w_{i-1}$ has at least $2 \vep_1 d + T$ neighbors in $G'$ that are not in $S'_{T - i}$. Let us count the number of neighbors $u$ of $w_{i - 1}$ in $G'$ for which $(w_1, \cdots, w_{k - 1}, u)$ is not a good path. There are at most $\vep_1 d$ such $u$ for which
\[
u \in V( P'' \cup \{w_1, \cdots, w_{i - 1}, u \})
\]
There are at most $\vep_1 d$ such $u$ for which
\[
P'' \cup \{w_1, \cdots, w_{i - 1}, u \}
\]
is a weak terminal path. And there are at most $T - 1$ such $u$ for which
\[
C(w_{i - 1}, u) \in C(P'' \cup \{w_1, \cdots, w_{i - 1} \})
\]
All other such $u$ extend to good paths. So there are at most $2 \vep_1 d + T - 1$ neighbors of $w_{i -1}$ which don't extend to good paths. Thus we can pick $w_i \not \in S'_{T - i}$ such that $(w_1, \cdots, w_i)$ is a good path.

\medskip

We end up with a path that extends $P'$ by length $T$ that intersects $B$ at most $\vep_4 T/2 - 1$ times which is a contradiction so we have proven the claim.

\bigskip

Since $V(G'') \subseteq S'_T$,
\[
|V(G'') | \leq |S'_T| \leq (3/\vep)^T K_2 d
\]
So we can take $K_3 = (1/\vep) (3/\vep)^T K_2$.
\end{proof}
\end{section}
\begin{section}{The Dense Case}
In this Section we prove the remaining Lemmas stated in Section \ref{proof_overview}. All these Lemmas have analogues in \cite{BFMPY}. The main difference is that their Lemmas were suited for the regime where
\[
	\delta^{+ }(G) = \Omega(n) \text{ and } \sup_{v \in V(G)} |\deg^{+}(v) - \deg^{-}(v) | = o(n)
\]
But since we wish to apply Theorem \ref{sec_1_main}, we prove variations on these Lemmas so as to only require
\[
	\delta^{+ }(G) = \Omega(n)
\]
This adaptation is not too difficult and proofs are essentially the same as in \cite{BFMPY}. Because of this, we do not prove the Lemmas in too much detail.

\medskip

\textit{Restatement of Lemma \ref{expander_finder}}. For any $\delta, \tau, \alpha$, there exists $\nu$ such that for every positive integer $n$ and every $n$-vertex digraph $G$. Suppose that $\delta^{+}(G) \geq \alpha n$. Then there exists a non-empty induced subgraph $G'$ of $G$ satisfying:
\begin{enumerate}
\item[A1] $G'$ is a robust $(\nu, \tau)$-out-expander.
\item[A2] Every vertex $v \in V(G')$ satisfies $\deg^{+}_{G'}(v) \geq \deg^+_{G}(v) - \delta n$
\end{enumerate}
\begin{proof}
First define a sequence of parameters, $\nu \ll \gamma_0 \ll \gamma_1 \ll \cdots \ll  \gamma_t \ll \tau, \delta, \alpha$ where $t = \lfloor \log_{1 - \tau/2}(\alpha/2) \rfloor$. Then we attempt to define a sequence of sets $U_0, U_1, \cdots $ where, for each $i$, $U_{i + 1}$ is a subset of $U_i$ such that $\tau |U_i|/2 \leq |U_{i + 1} | \leq (1 - \tau) |U_i|/2 $ and
\[
\partial^+ U_{i + 1} \leq \gamma_{i + 1} n^2
\]
Let $t$ be the minimal index for which we cannot find such a $U_{t + 1}$.
\begin{claim} $|U_i| \geq \alpha n /2$ for all $1 \leq i \leq t$.
\end{claim}
\begin{proof} Suppose this is not true and $t_0$ is the minimal index for which $|U_{t_0}| < \alpha n/2$. Then $|U_{t_0}| \geq \tau \alpha n/4$. So,
\[
\partial^+ U_{t_0} \geq \sum_{u \in U_{t_0} } \deg^+ (u) - |U_{t_0}|  ( |U_{t_0}| - 1 ) \geq |U_{t_0} | ( \alpha n - \alpha n/2 ) \geq \tau \alpha^2 n^2/8 > \gamma_{t_0} n^2
\]
which is a contradiction.
\end{proof}
We let
\[
U'_t = \{ v \in U_t : |N^{+} (v) \cap (V \setminus U_t) | < \sqrt { \gamma_{t - 1 } } n \}, Y_t = U_t \setminus U'_t.
\]
and define $G' = G[U'_t]$. We first check that \textbf{A2} holds. Note that
\[
|Y_t| \leq 1/(\sqrt \gamma_{t-1} n) \sum_{j = 1}^{t - 1} \gamma_j n^2 \leq 2 \sqrt{ \gamma_{t - 1}} n
\]
Therefore if $v \in U'_t$,
\[
\deg_{G'}^+ (v) \geq \deg_G^+ (v) - \sqrt{\gamma_t} n - |Y_t| \geq \deg_G^+(v) - 3 \sqrt{\gamma_{t- 1} }n \geq \deg_G^+(v) - \delta n
\]
This verifies \textbf{A2}. Now, let $W$ be a subset of $U'_t$ with $ \tau |U'_t| \leq |W| \leq (1 - \tau) |U'_t|$. Then,
\[
\partial_{G'}^+ (W) \geq \gamma_t n^2 - |Y_t| n \geq \gamma_t n^2/2
\]
Therefore, if we let $s = |\{u \in U'_t \setminus W : | n^{-} (u) \cap W | \geq \nu n \}| $, then
\[
(|U_t| - s) \nu |U_t| + s |W| \geq \gamma_t n^2 /2
\]
So $s \geq \nu |U_t|$. This verifies \textbf{A1} and so we have completed the proof.

\end{proof}
\medskip
\medskip
\medskip

\textit{Restatement of Lemma \ref {path_connecter}}. Let $\nu, \tau, \alpha, p, 1/M \leq 1$ be positive constants satisfying $\nu + \tau \leq \alpha$. Then for all sufficiently large $n$, the following holds: Let $G$ be a properly edge-colored directed graph on $n$ vertices. Suppose that $G$ is a robust $(\nu, \tau)$-out-expander with $\delta^+(G) \geq \alpha n$. Let $V_0, C_0$ be independent $p$-random subsets of $V(G), C(G)$, respectively. Then with high probability $1 - o_n(1)$ the following all hold
\begin{enumerate}
\item[ \textbf{B1}] $|V_0| \leq 2pn$
\item[ \textbf{B2}] Every vertex $v \in V(G)$ satisfies
\[
\deg^+ (v, V \setminus V_0 ; C \setminus C_0 ) \geq (1 - 3p) \deg^+(v)
\]
\[
\deg^{+} (v, V_0, C_0) \geq  p^2 \deg^{+}(v)/3
\]

\item[ \textbf{B3}] For any two distinct vertices $u, v \in V(G)$ with $\deg^{-}(v) \geq  (\tau + \nu) n$ and for any vertex subset $V_1 \subseteq V_0$ and color subset $C_1 \subseteq C_0$ each of size at most $M$, there exists a rainbow directed path of length at most $\nu^{-1} + 1$ from $u$ to $v$ in $G$ whose internal vertices are in $V_0 \setminus V_1$ and whose colours are in $C_0 \setminus C_1$.
\end{enumerate}
\begin{proof} As long as $n$ is sufficiently large, \textbf{B1} and \textbf{B2} both hold by the chernoff bound.

\medskip

For every vertex $u \in V(G)$, let $N_{0, u} = N^+(u)$ and for $1 \leq i \leq \nu^{-1}$, let $N_{i, u} = N_{i - 1, u} \cup RN^+_{\nu, G} (N_{i - 1, u})$.

\medskip

For vertices $u, v \in V(G)$ with $v \in RN^+_{\nu/2} (N_{0, u} )$ let
\[
Y_{u, v, 0} = \{ w \in N_{0, u} : v \in N^+(w), C(u, w) \in C_0, C(w, v) \in C_0, w \in V_0 \}
\]
and for $1 \leq i \leq \nu^{-1}$, for vertices $u, v$ with $v \in RN^+_{\nu/2} ( N_{i, u} \setminus N_{0, u} )$, let
\[
Y_{u, v, i} = \{ w \in N_{i, u} \setminus N_{0, u} : v \in N^+(w), C(w, v) \in C_0, w \in V_0 \}
\]
By chernoff bound and union bound $Y_{u, v, i} \geq 3(M  + 1) + 2\nu^{-1} + \sqrt n $ for all pairs of vertices $u$ and $v$ on which $Y_{u, v, i}$ is defined. (We note that the Chernoff bound gives a much stronger inequality but this is all we need). 

\medskip

Now fix a pair $C_0, V_0$ for which \textbf{B1},\textbf{B2} and the above bounds on $Y_{u, v, i}$ all hold. We will show that \textbf{B3} holds with high probability as long as $n$ is sufficiently large. For vertices $u, w \in V(G)$, let $C_{u, w} = \{ v \in V(G):  C(u, v) = C(v, w) \}$ and let
\[
H_{u} = \{ w \in V(G) :  | C_{u, w } | \geq \sqrt n \}
\]
and for $v \in V(G)$, let
\[
F_{v} = \{ w \in V(G) : v \in H_w \}
\]
Note that since $C$ is a proper edge coloring, $|H_{u}| \leq \sqrt n$ and $|F_{v}| \leq \sqrt n$. 
\begin{claim} Let $v \in V(G)$ with $\deg^{-}(v) \geq 3(\tau + \nu)$. For all subsets $C_1 \subseteq C_0$, $V_1 \subseteq V_0$ each of size at most $M + 1$ and for all $u \in V \setminus (F_v \cup V_1)$ there exists a path of length at most $\nu^{-1}$ from $u$ to $v$ with colors and internal vertices in  $G[ V_0 \setminus V_1, C_0 \setminus C_1]$.  
\end{claim}
Before proving our claim, let us see why it implies \textbf{C3}. Let $u$ and $v$ be an arbitrary pair of vertices with $\deg^{-}(v) \geq (\tau + \nu)n$ and $V_1 \subseteq V_0$ and $C_1 \subseteq C_0$ be sets of size at most $M$. Since $N^+ (u, V_0, C_0) \geq p^2 \deg^{+}(v)/3 \geq p^2 \alpha n/3$, there exists a $w \in N^+(u, V_0 \setminus (V_1 \cup F_v), C_0 \setminus C_1)$. Now we may apply the above claim to find a path from $w$ to $v$ which avoid colors $C_1 \cup \{ C(u, w) \}$ and avoids vertices $V_1 \cup \{u\}$. We then pre-append $u$ to the resulting path.

\medskip

\textit{Proof of Claim.} Since $G$ is a $(\nu, \tau)$-out-expander and $\alpha \geq \tau$, $|N_{0, u}| \geq \tau$ and for $1 \leq i \leq \nu^{-1}$,
\[
|N_{i, u}| \geq \min \{ (1 - \tau)n, |N_{i - 1, u} | + \tau n \}
\]
Therefore $|N_{\nu^{-1}, u} - 1| \geq (1 - \tau) n$. Since $ \deg^- (v) \geq (\tau + \nu) n$,
\[
v \in  N_{\nu^{-1}, u} 
\]
We then do the following algorithm to find a path from $u$ to $v$. Define a reverse path to be a sequence of vertices $(w_1, \cdots, w_k)$ such that $(w_k, \cdots, w_1)$ is a path in $G$. At all steps in our algorithm we will have a vertex $v' \in V(G) \setminus H_{u}$, a color set $C$ with $|C| < (M + \nu^{-1})$, a reverse path $Q$ ending at $v'$ with $C(Q) \subseteq (C \setminus C_1)$, and an index $i_0$ such that $v' \in N_{i_0, u}$.

\medskip

At each step of our algorithm we will either output a valid path satisfying $\textbf{B3}$ or decrease $i_0$. We begin with $v' = v$, $Q = \{v\}$, $C = C_1$, and $i_0$ the smallest positive integer for which $v \in N_{i_0, u}$. Then while $i_0 > 0$ we do the following:

\medskip

Since $v' \in N_{i, u}$, we have either $v' \in RN^+_{\nu/2} (N_{0, u} )$ or $v' \in RN^+_{\nu/2} ( N_{i_0 - 1, u} \setminus N_{0, u} )$.

\medskip

If the former is the case then by assumption $Y_{u, v', 0} \geq 3(M + 1) + 2\nu^{-1} + \sqrt n$ so there is some $w \in Y_{u, v', 0}$ such that $w \not \in V_1$, $C(u, w) \not \in C$, $C(w, v') \not \in C$, and $C(u, w) \neq C(w, v')$. Then we can update $Q$ to $Q \cup \{v', u\}$ and the path obtained by reversing the vertices in $Q$ satisfies \textbf{B3}.

\medskip 

On the other hand, suppose that $v' \in RN^+_{\nu/2} (N_{i_0 - 1, u } \setminus N_{0, u} )$. Then since $|Y_{u, v', i_0}| \geq 3(M + 1) + 2 \nu^{-1} + \sqrt n$, there is some $w \in Y_{u, v', i_0} \setminus H_u$ such that $C(w, v') \not \in C$ and $w \not \in V_1$. We then add $C(w, v')$ to $C$, append $w$ to $Q$, set $v' = w$ and set $i_0$ to be the smallest positive integer $i$ for which $w \in N_{i, u}$. It is easy to check that all requirements hold. So we are done.
\end{proof}
\textit{Restatement of Lemma \label{path_forest}} Let $\vep > 0$ and $G$ be an $n$-vertex properly colored directed graph  with $\delta^+(G) \geq d$ where $1/d \ll \vep$. Then for every $\vep > 0$ and vertices $V = \{v_0, \cdots, v_{4/\vep^2} \}$, there exists a family of vertex-disjoint, color-disjoint paths $\mathcal P = \{ P_0, \cdots, P_{4/\vep^2} \}$ such that the total length of all paths is at least $(1 - \vep) d$ and the first vertex in $P_i$ is $v_i$ for all $i$.
\begin{proof}
Let $t := 4/\vep^2$ and let $P_0, \cdots, P_{t}$ be such of set of paths of maximum possible length and assume towards contradiction that the total length is less than $(1 - \vep)d$. Then there must exist a subset of these paths $\mathcal P' = P'_1, \cdots, P'_{\sqrt t}$ such that the total length of these paths is at most $\vep d/2$. Let $H := \{h_0, \cdots, h_{2/\vep} \}$ be the endpoints of these paths. And let $\mathcal P^* = \mathcal P \setminus \mathcal P'$. Let $C_0$ be the set of colors not used in $\mathcal P$ and let $G_0$ be the subgraph with these colors and let $N_0$ be the neighborhood of $h_0$ in $G_0$. Additionally set $E_0 = \emptyset$. For every $i$, we will define $G_i, C_i, N_i, E_i$ so that $N_i \subseteq V(\mathcal P)$, $|N_i| \geq \vep (i + 2) /2 $ and $N_i = N^+_{G_i}(h_i)$. It is easy to see that $N_0$ satisfies these as otherwise we would be able to extend $P_1$. We define $G_i, C_i, N_i$ inductively as follows:
\begin{itemize}
	\item We let $E'_{i + 1} = \{ (u, v) \in \mathcal P^{*} : v \in N_{i} \setminus (P' \cup V) \}$ and set $E_{i + 1} = E_{i} \cup E'_{i + 1}$.
	\item We let $C'_{i + 1}$ be the set of colors of edges in $E_{i + 1}$
		There are at least $|N_i| - \vep d/2$ vertices in $N_i$ that are not in $V \cup V(\mathcal P')$. As we will see, $|N_i| \geq \vep(i + 3) d/2$ and so since all the $P_i$ are color disjoint we have $|C'_{i + 1}| \geq |N_i| - \vep d/2 \geq \vep (i + 1) d/2$. Let $C_{i + 1} = C'_{i + 1} \cup C_i$. Since $C_0$ contains no colors in any of the $P_i$, we have $|C_{i + 1}| \geq |C_0| + |C'_{i + 1}| \geq |C(G)| - d + \vep(i + 3) d/2$.
	\item We define $G_{i + 1}$ to be the subgraph of $G$ with colors in $C_{i + 1}$
	\item We define $N_{i + 1}$ to be the neighborhood of $h_{i + 1}$ in $G_{i + 1}$. Since $h_{i + 1}$ has degree at least $d$ and $|C_{i + 1}| \geq |C(G)| - d + \vep (i + 3)d/2$, $|N_{i + 1}| \geq \vep (i + 3) d /2$.
\end{itemize}
All that remains to verify is that $N_i \subseteq V(\mathcal P)$. Suppose towards contradiction it is not and $i$ is the minimum such index. Then there is a neighbor $x \not \in V( \mathcal P)$ of $h_i$ such that $C(h_i, x) \in C_i$. We add the edge $(h_i, x)$ to $\mathcal P$.

\medskip 

Then we define a sequence $i = j_0 > j_1 > \cdots > j_k = 0$ and sequence of colors $c_0, \cdots, c_k$ where $c_0 = C(h_i, x)$. We define them inductively while doing the as following operations to $\mathcal P$:
For $1 \leq \ell \leq k$ we let $j_{\ell}$ be the minimal index for which $c_{\ell - 1} \in C_{j_\ell}$. Suppose that $j_{\ell} > 0$. Then by definition of $C_{j_\ell}$, there is an edge $(u_\ell, v_\ell) \in E_{j_\ell}$ with color $c_{\ell - 1}$. Since we chose $j_\ell$ to be the minimal index, that edge is in $E'_{j_{\ell}}$. Therefore, $(h_{j_{\ell} - 1}, v_\ell)$ is an edge with color in $C_{j_{\ell} - 1}$. We let $c_{\ell} = C( h_{j_{\ell} - 1}, v_\ell)$ and we modify $\mathcal P$ by adding the edge $(h_{ {j_\ell} - 1}, v_{\ell})$ and deleting the edge $(v_\ell, u_\ell)$ already in $\mathcal P$.

\medskip

We claim that at step $1 \leq \ell \leq k - 1$, $\mathcal P$ consists of $t$ paths containing paths $P''_1, \cdots, P''_{\sqrt t} \in \mathcal P$ where $P''_{s} = P'_s$ for $s < j_{\ell}$ and $P''_s$ contains $P'_s$ as a prefix for $s \geq j_{\ell}$. Additionally, every edge $e \in E_{j_\ell - 1}$ is in $\mathcal P$. Finally $\mathcal P$ is rainbow except for the edge $(h_{j_\ell}, u)$ and the number of edges in $\mathcal P$ is one greater than it was at the beginning of the process.
This is easy to verify by induction.

\medskip

Finally, we reach $\ell = k$. At this step the edge $C(h_{j_\ell}, u) \in C_0$ so after adding this edge and deleting $(w, u)$, $\mathcal P$ is rainbow and contains strictly more edges at the beginning. And $V_0, \cdots V_{t}$ remain the starting vertices of the paths in $\mathcal P$. This gives a contradiction so we get $N_i \subseteq V(\mathcal P)$ for all $i$. But this means for $i = \sqrt{t}$, $|C_{i + 1}| > |C(G)|$ which is a contradiction.
\end{proof}
\end{section}

\begin{section}{Proof of Theorem \ref{main}} \label{main_proof}
\textit{Proof of Theorem \ref{main}.} Suppose towards contradiction that Theorem \ref{main} is false so there exists an $\vep > 0$ and arbitrarily large $d$ for which there exists a $d$-regular properly edge colored directed graph $G$ with no path of length $(1 - \vep) d$.

\medskip

Fix $\gamma , \vep_1 \ll \vep$. Apply Theorem \ref{sec_1_main} to get a constant $K$, a rainbow path $P$, and an induced subgraph $G' \subseteq G[V \setminus V(P), C \setminus C(P)]$ such that for every $u \in G'$, there exists a rainbow path extending $P$ by length at most $K$ that ends at $u$, $V(G') \leq K d'$, and $\delta^+(G') \geq (1 - \gamma)d'$
\medskip
Let $\delta \ll \tau \ll \gamma, 1/K$, and $\alpha = 1/(2K)$ and apply Lemma \ref{expander_finder} with the $\delta$ replaced by $\delta/K$ to get a $(\nu, \tau)$-out-expander $G''$ with $\delta^+(G'') \geq (1 - \gamma - \delta) d'$. We may assume that $\nu \ll \tau$.

\medskip

Let $P'$ be a minimal length path extending $P$ with an end in $G''$. By Lemma \ref{sec_1_main}, $|P'| \leq |P| + K$. Apply Lemma \ref{path_connecter} with $\tau \ll p^2 \ll \gamma$ and $M = 4 \nu^{-1} /\gamma^2$ to get $p$-random subsets $V_0$, $C_0$ of $V(G'')$, $C(G'')$ respectivly. We may assume that no vertices or colors in $P' - P$ occur in $V_0$ or $C_0$ respectivly since this happens with positive constant probability. Let $H = G''[V(G'') \setminus V_0, C(G'') \setminus C_0]$.

\medskip

By Lemma \ref{path_connecter}, $\delta^+ (H) \geq (1 - \gamma - \delta - 3p) d'$ so the average in-degree of $H$ is at least $(1 - \gamma - \delta - 3p) d'$. Since $\Delta^{-}(H) \leq d \leq d'/\gamma$ there are at least
	\[
	|V(H)| \gamma (1 - \gamma - \delta - 3p - K(\nu + \tau) ) \geq d' \gamma (1 - \gamma - \delta - 3p - K(\nu + \tau) ) \geq 4/\gamma^2
	\]
vertices $v \in V(H)$ with
	\[
\deg^{-}(v) \geq  (\nu + \tau) K d'
	\]
So we can find vertices $\{v_1, \cdots, v_{4/\gamma^2} \} \in V(H)$ each with in-degree at least $ [3(\nu + \tau)/p]K d'$. Let $v_0 = \End(P')$ and apply Lemma \ref{path_forest} replacing $\vep$ with $\gamma$ to get a family of path $\mathcal P = \{P_0, \cdots, P_{ 4/\gamma^2 } \}$ with total length at least $(1 - 2 \gamma - \delta - 3p)d'$ such that the first vertex of $P_i$ is $v_i$ for all $0 \leq i \leq 4/\gamma^2$.

\medskip

Since
	\[
	\deg^{-} (v_i) \geq (\nu + \tau) K d' \geq  (\nu + \tau) |V(G''')|
	\]
for every $1 \leq i \leq 4/\gamma^2$, by Lemma \ref{path_connecter}, we may in sequence connect $\End(P_{i - 1})$ to $V_i$ with a rainbow-path of length at most $\nu^{-1} + 1$ using colors in $C_0$ and internal vertices in $V_0$, at each step removing the set of vertices and colors used in previous paths. This is possible because in total we end up removing at most $4\nu^{-1}/\gamma^2 = M$ vertices and colors. Thus we end up with a rainbow path of length at least
	\[
	|P| + (d - |P|) (1 - 2 \gamma - \delta - 3p) \geq (1 - 2 \gamma - \delta -3p) d \geq (1 - \vep) d
	\]
So we have completed the proof.
\end{section}
\begin{section}{Concluding Remarks}
Our techniques are capable of weakening the assumtions of Theorem \ref{main} from being $d$-regular to $\delta^+(G) \geq d$ and $\Delta^-(G) = O(d)$. It would be nice to remove the condition $\Delta^-(G) = O(d)$ so that we could prove every properly edge-colored directed graph contains a path of length $\delta^+(G)$. Despite using $\Delta^-(G) = O(d)$ frequently in the proof of Theorem \ref{main}, it is possible to avoid all of these except in the proof of Lemma \ref{short_ends}. However this seems to be a major roadblock. For example consider consider the $d$-ary tree of height $d/2$ with every edge a unique color and directed away from the root. Then, for every leaf $u$, letting $P$ be the path from the root of the tree to $u$, draw $d/2$ edges from $u$ to other leaves using colors $C(P)$ and draw and edge from $u$ to each vertex in $V(P)$ using new colors. The resulting graph can easily be made to be properly edge-colored and has minimum out-degree $\delta^+(G)$ but
\begin{itemize}
\item The longest rainbow path starting at the root has length $d/2$.
\item Lemma \ref{short_ends} fails horribly on $G$.
\end{itemize}
Thus, to remove the $\Delta^-(G) = O(d)$ requirement, it seems that major new ideas are needed.
\end{section}
\section{Acknowledgments}
We would like to thank Jaques Verstraete for helpful conversations. We would also like to thank Bryce Frederickson for suggesting a way to improve the constants in Lemma \ref{sec_1_main}.


\begin{thebibliography}{99}
\bibitem{Al} N. Alon, \textit{Combinatorial Nullstellensatz}, Combin. Probab. Comput. 8 (1999), no. 1-2, 7-29.
\bibitem{APS} N. Alon, A. Pokrovskiy, and B. Sudakov, \textit{Random subgraphs of properly edg-coloured complete graphs and l ong rainbow cycles}, Israel J. Math. 222 (2017), 317-331.
\bibitem{AG} B. Alspach and H. Gavlas, \textit{Cycle decompositions of $K_n$ and $K_{n - 1}$}, J. Comb. Theory Ser. B 81 (2001), no . 1, 77-99.
\bibitem{AGSH} B. Alspach, H. Gavlas, M. \u Sajna, and V.H., \textit{Cycle decompositions IV: complete directed graphs and fixed length directed cycles}, J. Comb. Theory Ser. A 103 (2003), no. 1, 165-208.
\bibitem{AL} B. Alspach and G. Liversidge, \textit{On strongly sequenceable abelian groups}, Art Discrete Appl. Math. (2020).
\bibitem{An} L. D. Andersen, \textit{Hamilton circuits with many coulours in properly edge-colored complete graphs}, Math. Scand. 64 (1989), no. 1, 5-14
\bibitem{ADMS} D. S. Archdeacon, J. H. Dinitz, A. Mattern, and D. R. Stinson, \textit{On partial sums in cyclic groups}, J. Combin. Math. Combin. Comput. 98 (2016), 327-342.
\bibitem{BCR} J. Babu, L. S. Chandran, and D. Rajendraprasad, \textit{Heterochromatic paths in edge colored graphs without small cycles and heterochromatic-triangle-free graphs}, Eur. J. Comb. 48 (2015), 110-126
\bibitem{BM} J. Balogh and T. Molla, \textit{Long rainbow cycles and Hamiltonian cycles using many colors in properly edge-colored complete graphs}, Eur. J. Comb. 79 (2019), 140-151
\bibitem{BJ} S. T. Bate and B. Jonoes, \textit{A review of uniform cross over designs}, J. Statist. Plann. Inference 138 (2008), no. 2, 336-351.
\bibitem{Go} B. Gordon, \textit{Sequences in groups with distinct partial products}, Pacific J. Math. \textbf{11} (1961), 1309-1313.
\bibitem{Ke} A. Keedwell, P. Cameron, J. Hirschfeld, and D. Hughes, \textit{Sequenceable groups: a survey} LMS Lecture Notes \textbf{49} (1981), 205-215.
\bibitem{MP} A. M\''{u}yesser and A. Pokrovskiy, \textit{A random Hall-Paige conjecture}, Invent. Math. (to appear), arXiv::2204.09666 (2022).
\bibitem{Gr} R. L. Graham, \textit{On sums of integers taken from a fixed sequence}, Proceedings of the Washington State University Conference on Number Theory (Washington State Univ., Pullman, Wash., 1971), Washington State University, Department of Mathematics, Pi Mu Epsilon, Pullman, WA, 1971, pp. 22-40.
\bibitem{Kr} N. Kravitz, \textit{Rearranging small sets for distinct partial sums}, (2024), arXiv:2409.07403
\bibitem{Sa} W. Sawin, MathOverflow (2015), comment on the post "Ordering subsets of the cyclic group to give distinct partial sums".
\bibitem{BK} B. Bedert and N. Kravitz, \textit{Graham's rearrangement conjecture beyond the rectification barrier}, (2024), Isr. J. Math. (to appear), arXiv:2409.07403.
\bibitem{CMPP} S. Costa, F. Morini, A. Pasotti, and M. A. Pellegrini, \textit{A problem on partial sums in abelian groups,} Discrete Math. \textbf{341} (2018), no. 3, 705-712.
\bibitem{CDFO} S. Costa, S. Della Fiore, M. Ollis, and S. Z. Rovner-Frydman, \textit{On sequences in cyclic groups with distinct partial sums}, arXiv:2204.16658 (2022)
\bibitem{Ha} G. Hahn, Un jeu de coloration, Regards sur la theorie des graphs, Actes du Colloque de Cerisy \textbf{12} (1980).
\bibitem{MM} M. Maamoun and H. Meyniel, \textit{On a problem of G. Hahn about Coloured hamiltonian paths in $K_{2t}$,} Discrete Math. \textbf{51} (1984), no. 2, 213-214.
\bibitem{Sc} L. Schrijver, \textit{Rainbow paths in edge-coloured regular graphs,} Manuscript.
\bibitem{JPS} D. Johnston, C. Palmer, and A. Sakar, \textit{Rainbow Tur\'{a}n numbers of paths and other trees,} Australas. J. Combin. \textbf{78} (2020), 61-72.
\bibitem{BFMPY} M. Bucic\'{c}, B. Frederickson, A. M\"{u}yesser, A. Pokrovskiy, \textit{Towards Graham's rearrangement conjecture via rainbow paths}, arXiv:2503.01825.
\bibitem{BBK} Benjamin Bedert, Matija Buci\c, Noah Kravitz, Richard Montgomery, Alp M\"yesser. On Graham's Rearrangement Conjecture over $F^n_2$, arXiv:2508.18254 (2025)
\bibitem{BBK} Bedert, B., Buci\c, M., Kravitz, N., Montgomery, R., M\'yesser, A. (2025). On Graham's Rearrangement Conjecture over $F^n_2$. arXiv:2508.18254.
\bibitem{PS} Pham, H. T., Sauermann, L. (2026). On Graham's Rearrangement Conjecture. arXiv:2602.15797.
\end{thebibliography}
\end{document}